\documentclass{amsart}
\usepackage{amsmath}
\usepackage{amsthm}
\usepackage{amssymb}
\usepackage{amsfonts}
\usepackage{enumerate}
\usepackage{mathrsfs,stmaryrd}
\usepackage{algpseudocode,algorithm}
\usepackage{latexsym,amscd}
\newtheorem{theorem}{Theorem}
\newtheorem{lemma}[theorem]{Lemma}
\newtheorem{corollary}[theorem]{Corollary}

\newtheorem{proposition}[theorem]{Proposition}
\theoremstyle{definition}
\newtheorem{definition}[theorem]{Definition}

\newtheorem{example}[theorem]{Example}
\newtheorem{remark}[theorem]{Remark}

\def\Z{{\mathbb{Z}}}
\def\N{{\mathbb{N}}}
\def\Q{{\mathbb{Q}}} 
 
\def\C{{\mathbb{C}}}
\def\K{{\mathbb{K}}}
\def\F{{\mathbb{F}}}

\def\m{{\mathfrak{m}}}
\def\n{{\mathfrak{n}}}
\def\p{{\mathfrak{p}}}
\def\q{{\mathfrak{q}}}
\def\a{{\mathfrak{a}}}
\def\b{{\mathfrak{b}}}

\def\LT{\mathop{\mathrm{LT}}\nolimits}
\def\LC{\mathop{\mathrm{LC}}\nolimits}
\def\LM{\mathop{\mathrm{LM}}\nolimits}

\def\Hom{\mathop{\mathrm{Hom}}\nolimits}
\def\gr{\mathop{\mathrm{gr}}\nolimits}

\def\supp{\mathop{\mathrm{supp}}\nolimits}
\def\Syz{\mathop{\mathrm{Syz}}\nolimits}

\def\notdivide{\setbox1=\hbox{$|$\llap{\hbox{/}\kern-0.75pt}}\mathrel{\box1}}
\providecommand{\abs}[1]{\lvert#1\rvert}

\title[An algorithm for computing Hilbert--Samuel multiplicities]
{An algorithm for computing the Hilbert--Samuel multiplicities and reductions of zero-dimensional ideals of Cohen--Macaulay local rings}

\author{Takafumi Shibuta}
\address{Institute of Mathematics for Industry, 744 Motooka, Nishi-ku. Fukuoka, Japan}
\email{shibuta@imi.kyushu-u.ac.jp}

\author{Shinichi Tajima}
\address{Graduate School of Science and Technology, Niigata University, Niigata 950-2181, Japan}
\email{tajima@emeritus.niigata-u.ac.jp}
\keywords{
Hilbert--Samuel multiplicity; Integral closure; Formal power series ring
}
\begin{document}
\maketitle
\begin{abstract}
In this paper, we present an algorithm for computing the minimal reductions of 
$\mathfrak{m}$-primary ideals of Cohen--Macaulay local rings. 
Using this algorithm, we are able to compute the Hilbert--Samuel multiplicities and 
solve the membership problem for the integral closure of $\mathfrak{m}$-primary ideals. 
\end{abstract}
\section{Introduction}
Let $(R,\m,K)$ be a Noetherian local ring of dimension $d\ge 0$, and 
$J \subset R$ a $\mathfrak{m}$-primary ideal. 
{
Let $G=\gr_J(R):=\bigoplus_{k=0}^\infty J^k/J^{k+1}$ be the associated graded ring of $J$. 
We denote by $H_G(k):=\ell(J^k/J^{k+1})$ the Hilbert function of $G$, and by $\chi_J(k)=\ell(R/J^k)=\sum_{i=0}^k H_G(i)$ the Hilbert--Samuel function of $J$ 
were $\ell(M)$ denotes the length of an $R$-module $M$. 
The Hilbert--Samuel polynomial $P_J(t)\in \Q[t]$ is the polynomial of degree $d-1$ satisfying $P_J(k)=\chi_J(k)$ for $k\gg 0$. 
The \emph{Hilbert--Samuel multiplicity} $e_R(J)$ of $J$ is $d!$ times the leading coefficient of $P_J(t)$.} 
By definition, 
\[
e_R(J)=\lim_{k\to \infty}\frac{d!}{k^d}\ell(R/J^{k}). 
\] 
The Hilbert-Samuel multiplicity is one of the most important invariants associated to a $\mathfrak{m}$-primary ideal of a local ring. 
For example, 
the Hilbert--Samuel multiplicities has deep relation with integral closure of ideals. 
If $R$ is formally equidimensional, $e_R(J_1)=e_R(J_2)$ if and only if $J_2\subset \overline{J}_1$ for $\m$-primary ideals $J_1\subset J_2\subset R$ \cite{Rees}. 

{If $I$ is generated by polynomials, it is well-known that the Hilbert--Samuel multiplicity $e_R(\m)$ of $\m\subset R=K\llbracket x_1,\dots,x_n \rrbracket/I$ 
can be computed using the tangent cone \cite{Mora} (see also \cite{Singular} Proposition 5.5.7).} 
In \cite{MR}, Mora and Rossi gave an algorithm for computing the Hilbert--Samuel multiplicities of primary ideals generated by polynomials. 
If $R$ is the localization $S_\m$ of a polynomial ring $S=K[x_1,\dots,x_n]$ at a maximal ideal $\m\subset S$, 
and $J\subset S$ is a $\m$-primary ideal, $\gr_{JR}(R)$ is isomorphic to the associated graded ring $\gr_J(S)=\bigoplus_{k=0}^\infty J^k/J^{k+1}$ of $J$. 
Note that $\gr_{JR}(R)$ is not isomorphic to $\gr_J(S)$ if $J$ is not $\m$-primary. 
In this case, we can compute an expression of $\gr_J(S)$ as a quotient ring of a polynomial ring module an ideal using the theory of Gr\"obner basis (\cite{MR} Theorem 1). 
Hence the Hilbert polynomial of $\gr_J(S)$ is computable, and thus the Hilbert--Samuel multiplicity $e_R(J)$ is also computable by using Gr\"obner basis. 

The purpose of this paper is to present an algorithm for computing the Hilbert--Samuel multiplicity $e_R(J)$ of $\m$-primary ideal of a local ring $(R,\m,K)$ 
when $R=K\llbracket x \rrbracket/I$ is a quotient ring of a formal power series $K\llbracket x\rrbracket=K\llbracket x_1,\dots,x_n\rrbracket$ and $R$ is Cohen--Macaulay. 
Our algorithm is applicable when $I$ and $J$ are generated by computable power series (Corollary \ref{infinite terms}). 
Here, we call a power series $f=\sum_{\alpha \in \Z_{\ge 0}^n}c_\alpha x^\alpha$ computable if the function $\alpha \mapsto c_\alpha$ is a computable function. 
In this case, algorithms of Buchberger's type (e.g. Gr\"obner basis, Mora's algorithm \cite{Mora}, Lazard's homogenization method \cite{Lazard}) are not applicable for computation. 
Our algorithm also has advantages in the case where $J=J'R$ is generated by a polynomial ideal 
$J' \subset K[x]$ and $J'$ is not $\langle x_1,\dots,x_n\rangle$-primary (Example \ref{ex}). 

{Our algorithm is based on the theory of reduction of ideals and local cohomology modules, 
and uses the algorithm proposed by Tajima--Nakamura--Nabeshima \cite{TNN} for computing local cohomology modules.} 
An ideal $Q\subset J$ is called a \emph{reduction} of $J$ if there exists $r>0$ satisfying $QJ^r=J^{r+1}$. 
It is known that $e(Q)=e(J)$ holds if $Q$ is a reduction of $J$. 
In this paper, we give an algorithm for computing reductions of $\mathfrak{m}$-primary ideals when $R$ is a Cohen-Macaulay ring of dimension $d$. 
Since it is known that $d$ generic linear combinations of a system of generators ideals generate a reduction {(\cite{Matsumura} Theorem 14.14)}, 
one can construct probabilistic algorithm for computing reductions of ideals if the coefficient field is infinite. 
However, this method is probabilistic, and the size of coefficients of the output can be large. 
If the coefficient field is finite, this method can not be applied. 
In our algorithm, we use a generic linear combinations with coefficients being variables. 
Thus our algorithm is deterministic, and applicable when the coefficient field is finite. We also give a way of computing reductions with small coefficients. 
As application of this algorithm, we can compute the Hilbert--Samuel multiplicities, 
and can solve the membership problem for integral closures. 
\section{Preliminaries}
In this paper, $S=K\llbracket x\rrbracket=K\llbracket x_1,\dots,x_n\rrbracket$ denotes a formal series ring in indeterminates $x=(x_1,\dots,x_n)$ over arbitrary field $K$. 
We denote by $\m$ the unique maximal ideal of $S$. 
We write $x^1=x_1\cdots x_n$ and $x^{\alpha+1}=x_1^{\alpha_1+1}\cdots x_n^{\alpha_n+1}$ for $\alpha=(\alpha_1,\dots,\alpha_n)$. 
For a ring $R$ and an $R$-module $M$, we denote by $\ell_R(M)=\ell(M)$ the length of $M$. 
For $B\subset M$, $\langle B \rangle_R=\langle B \rangle$ denotes $R$-submodule of $M$ generated by $B$. 
\subsection{Matlis duality}
Our algorithm is based of Matlis duality theorem. 
Here, we give a brief review of Matlis duality. 
Let $E=E_S$ the injective hull of the residue field $S/\m$ of $S$. 
It is known that $E$ is isomorphic to the top local cohomology group $H_{\m}^{n}(S)$ (\cite{BH} Proposition 3.5.4). 
The local cohomology module $H_{\m}^{n}(S)$ has a relative \v{C}ech cohomology representation 
\[
H_{\m}^{n}(S)\cong {K}[x_1^{-1},\dots,x_n^{-1}]\frac{1}{x_1\cdots x_n} 
\]
with an $S$-module structure defined by 
\[
x^{\alpha}\cdot \frac{1}{x^{\beta+1}}=
\left\{
	\begin{array}{ll}
	 \cfrac{1}{x^{\beta-\alpha+1}}
	& \mbox{ if } \beta-\alpha \in \Z^{n}_{\ge 0}, \\[3mm]
	0
	& \mbox{otherwise}. 
	\end{array}
\right.
\]
and its bilinear extension. 
In this paper, we identify $E$ with the above \v{C}ech cohomology representation of $H_{\m}^{n}(S)$, 
\[
E={K}[x_1^{-1},\dots,x_n^{-1}]\frac{1}{x_1\cdots x_n},
\]
and we call $\cfrac{1}{x^{\alpha+1}}$ a \emph{term} of $E$. 
An $S$-submodule of $E$ generated by terms is called a \emph{term module} in this paper. 
\begin{definition}
For an $S$-module $M$, we write $M^\vee:=\Hom_S(M,E)$. 
The functor $(-)^\vee=\Hom_S(-,E)$ is called the \emph{Matlis duality functor}. 
\end{definition}
Now, we recall the Matlis duality theorem. {See \cite{BH} Proposition 3.2.12, Theorem 3.2.13 for the proof. }
\begin{theorem}[Matlis]\label{Matlis}
Let $M$ be a Noetherian $S$-module, and $N$ an Artinian $S$-module. 
Then the following hold. 
\begin{enumerate}[{(}1{)}]
\item $S^\vee \cong E$, and $E^\vee\cong S$. 
\item $M^\vee$ is Artinian, and $N^\vee$ is Noetherian. 
\item There are natural isomorphism $M^{\vee\vee}\cong M$ and $N^{\vee\vee}\cong N$. 
\item {If $M$ is of finite length, then $\ell(M)=\ell(M^\vee)$. }
\end{enumerate}
\end{theorem}
For an ideal $J\subset S$, we can identify $(S/J)^\vee$ with the submodule $\{\eta \in E\mid J\eta =0\}$ of $E$. 
By Theorem~\ref{Matlis}, for $f\in S$, $f\in J$ if and only if $f\eta=0$ for all $\eta \in (S/J)^\vee$. 
\subsection{Integral closure}
Here, we recall some basic facts on the integral closure of ideals. See \cite{BH} Section 4, and \cite{HS} for details. 
Let $(R, \n, K)$ be a local ring of dimension $d$, and $J \subset R$ a $\n$-primary ideal. 
\begin{definition}
We say that $x\in R$ is \emph{integral } over $J$ if there exists $m\in \N$ and $c_i \in J^i$ such that 
\[
x^m+c_1x^{m-1}+\cdots+c_m=0. 
\] 
\end{definition}
\begin{theorem}\label{int cl}
Let $\mathcal{O}_{\mathbb{C}^d,O}$ be the ring of analytic germs at the origin $O\in \C^d$. 
Let $J=\langle {f}_1, \dots, {f}_r \rangle \subset \mathcal{O}_{\mathbb{C}^d,O}$ and $g\in \mathcal{O}_{\mathbb{C}^d,O}$. 
Then the following are equivalent. 
\begin{enumerate}[{(}1{)}]
\item\label{(i)} $g \in \overline{J}$. 
\item\label{(ii)} There exist a open neighborhood $U$ of $O$ and a constant $C>0$ such that for any $x\in U$, 
\[
\abs{g(x)} \le C(\abs{{f}_1(x)}+\cdots \abs{{f}_r(x)}). 
\]
\item $g\in H^0(X,\mathcal{O}_X(-F))$ where $X\to (\C^d,O)$ is a log resolution of $J$, and $F$ is the divisor $J\mathcal{O}_X=\mathcal{O}_X(-F)$. 
\end{enumerate}
\end{theorem}
See \cite{Te} for the equivalence of (\ref{(i)}) and (\ref{(ii)}). 
\begin{example}
Let $J=\langle x^2,y^2\rangle\subset \mathcal{O}_{\C^2,O}$, $g=xy$. Then $g\in \overline{J}$ since $g^2-c=0$ where $c=x^2y^2\in J^2$. 

As $\abs{xy}\le \frac{1}{2}(\abs{x^2}+\abs{y^2})$, $g$ and $J$ satisfy (2). 

The blowing-up $X\to \C^2$ of $\C^2$ at the origin is a log resolution of $J$. 
Then $X=U\cup V$ where 
$U=\{(x,y/x)\ \mid x\neq 0\}$ and $V=\{(x/y,y)\ \mid y\neq 0\}$. 
On $U$, $g=x^2\cdot {y/x} \in J\mathcal{O}_{U}=\langle x^2, x^2(y/x)^2\rangle=\langle x^2\rangle $, and  
on $V$, $g=y^2\cdot {x/y} \in J\mathcal{O}_{V}=\langle y^2(x/y)^2, y^2\rangle =\langle y^2\rangle$. 
Thus $g$ and $J$ satisfy (3). 
\end{example}
For simplicity, we consider a complete local ring $R=S/I$ of dimension $d$ where $S=K\llbracket x_1,\dots,x_n \rrbracket$. 
We denote by $\n$ the unique maximal ideal $\m/I$ of $R$. 

\begin{definition}
The Hilbert--Samuel multiplicity $e_R(J)$ of an $\n$-primary ideal $J\subset R$ is defined by 
\[
e_R(J)=\lim_{k\to \infty}\frac{d!}{k^d}\ell(R/J^{k}). 
\]
\end{definition}
\begin{definition}
Let $Q\subset J\subset R$ be ideals. $Q$ is said to be a \emph{reduction} of $J$ if 
$QJ^r=J^{r+1}$ for some integer $r>0$. 
\end{definition}
Reductions and Hilbert--Samuel multiplicities are useful tools for theory of integral closure. 
\begin{theorem}[\cite{Rees}]\label{cl and HS}
Assume that $R=S/I$ is equidimension. Let $J_1\subset J_2\subset R$ be $\n$-primary ideals. 
Then $e_R(J_1)=e_R(J_2)$ if and only if $J_2\subset \overline{J}_1$. 
\end{theorem}
By this theorem, 
we are able to solve the membership problem for the integral closure of an $\n$-primary ideal using  Hilbert--Samuel multiplicity. 
\begin{corollary}
Assume that $R=S/I$ is equidimension. 
For $f\in R$ and $\n$-primary ideal $J$, $f\in \overline{J}$ if and only if $e_R(J)=e_R(\langle J,f\rangle)$. 
\end{corollary}
Since Cohen--Macaulay rings are equidimensional, this criterion is applicable if $R$ is Cohen--Macaulay. 
\begin{proposition}\label{membership}
If $J_1$ is a reduction of $J_2$, then $e_R(J_1)=e_R(J_2)$. 
\end{proposition}
\begin{corollary}\label{CM case}
Assume that $R$ is Cohen--Macaulay. Let $J\subset R$ be an $\n$-primary ideal, 
and $Q\subset$ an $\n$-primary ideal generated by $d$ elements. 
Then $e_R(J)=\ell(R/Q)$ if and only if $Q$ is a reduction of $J$. 
\end{corollary}
It is known that $d$ generic liner combinations of a system of generators of $J$ generates a reduction of $J$. 
{A subset $U\subset K^n$ is called a {\em Zariski open set} if $U=K^n\backslash V_K(\a)$ for some $\a\subset K[t_1,\dots,t_n]$ 
where $V_K(\a)=\{\alpha \in K^n \mid f(\alpha)=0 \mbox{ for all } f\in \a\}$.}
\begin{theorem}\label{red}({See \cite{Matsumura} Theorem 14.14})
Let $J=\langle {f}_1,\dots,{f}_m \rangle$, and assume that $K$ is an infinite field. 
There exists a Zariski open set $U\subset K^{d\times m}$, such that for $(a_{ij})_{i,j} \in U$, 
$Q=\langle {g}_1,\dots, {g}_d \rangle$, ${g}_j=\sum_i a_{ij}{f}_i$ $(1\le j\le d)$, is a reduction of $J$. 
\end{theorem}
\section{Composition series with term orders}
Recall that $S=K\llbracket x\rrbracket=K\llbracket x_1,\dots,x_n\rrbracket$. 
In this section, we define term orders, and give composition series of $N\subset E$ with $\ell(N)<\infty$ uniquely determined by a given term order. 
We need this composition series later for constructing algorithms. 

\begin{definition}\label{def:term order}
A total order $\prec$ on the set of terms of $E$ is called \emph{term order} on $E$ if for any $\alpha,~\beta,~\gamma\in \Z_{\ge 0}^n$, the following conditions hold
\begin{enumerate}[{(}1{)}]
\item $\cfrac{1}{x^1}\preceq \cfrac{1}{x^{\alpha+1}}$ ,
\item $\cfrac{1}{x^{\alpha+1}}\prec \cfrac{1}{x^{\beta+1}}$ implies $\cfrac{1}{x^{\alpha+\gamma+1}}\prec \cfrac{1}{x^{\beta+\gamma+1}}$. 
\end{enumerate}
\end{definition}
In Gr\"obner basis theory, a total oder $\prec$ on the set of terms of the polynomial ring $K[x_1,\dots,x_n]$ 
is said to be a term order if for any $\alpha,~\beta,~\gamma\in \Z_{\ge 0}^n$, the following conditions hold
(1) $1\preceq x^\alpha$, and (2) $x^\alpha\prec x^\beta$ implies $x^{\alpha+\gamma}\prec x^{\beta+\gamma}$. 
Giving a term order defined in Definition \ref{def:term order} is essentially equivalent to giving a term order in meaning of Gr\"obner basis theory, 
thus it is easy to implement in a computer algebra system equipped with a package for Gr\"obner basis. 
It is known that term orders on the set of terms of the polynomial rings are well-ordering. 
Thus we have the following. 
\begin{proposition}
A term order on $E$ is a well-ordering. 
\end{proposition}
\begin{definition}
For $\eta=\sum c_\alpha \cfrac{1}{x^{\alpha+1}} \in E$, and a term order $\prec$ on $E$, we call
\[
\LT_\prec(\eta):=\max_\prec \left\{\cfrac{1}{x^{\alpha+1}} ~\middle|~ c_\alpha\neq 0 \right\}
\]
the \emph{leading term} of $\eta$ with respect to $\prec$. 
{We denote by $\LC_\prec(\eta)$ the coefficient of $\LT_\prec(\eta)$ in $\eta$, 
and set $\LM_\prec(\eta)=\LC_\prec(\eta)\LT_\prec(\eta)$. }
\end{definition}
Since any element of $E$ has only finitely many terms, the leading term is well-defined. 
For a subset $N\subset E$, $\LT_\prec(N)=\{\LT_\prec(\eta)\mid \eta \in N\}$. 
\begin{definition}
For an $S$-submodule of $E$, we call $N:_E\m:=\{\eta \in E\mid \m \eta \in N\}$ the \emph{colon module} of $N$. 
\end{definition}
\begin{lemma}\label{lem: colon}
If $\eta \in N:_E\m$ then 
$\LT_\prec(\eta) \in \LT_\prec(N):_E\m$. 
\end{lemma}
\begin{proof}
For any $1\le i\le n$, $x_i\eta \in N$. 
It is easy to show that $\LT_\prec(x_i \eta)=x_i\LT_\prec(\eta)$ if $x_i\LT_\prec(\eta)\neq 0$. 
Thus $x_i\LT(\eta)=0$ or $x_i\LT_\prec(\eta)\in \LT_\prec(N)$ for any $i$. 
This {shows} that $\LT_\prec(\eta) \in \LT_\prec(N):_E\m$. 
\end{proof}
The Matlis {dual} of $S/M$ where $M$ is a monomial ideals is easy to understand. 
A monomial not contained in a monomial ideal $M$ is called a \emph{standard monomial} of $M$. 
\begin{lemma}\label{monomial case}
Let $M\subset S$ be a monomial ideal with {the} minimal system of monomial generators 
$\{x^{\alpha_1},\dots,x^{\alpha_k}\}$, $\alpha_i\in \Z_{\ge 0}^n$, and let $N=(S/M)^\vee$. 
Let $\displaystyle \bigcap_{i=1}^r \left\langle x_1^{\beta^{(i)}_1},\dots,x_n^{\beta^{(i)}_r}\right\rangle$ be the irredundant irreducible decomposition of $M$, 
and let $\beta^{(i)}=(\beta^{(i)}_1, \dots,\beta^{(i)}_n)$. 
Then the following hold. 
\begin{enumerate}[{(}1{)}]
\item $\left\{\cfrac{1}{x^{\beta+1}} ~\middle|~ \mbox{ $x^{\beta}$ is a standard monomial of $M$} \right\}$ is a basis of $N$ as a $K$-vector space.\\ 
\item If $\ell(S/M)<\infty$, then $\left\{\cfrac{1}{x^{\beta^{(1)}}},\dots, \cfrac{1}{x^{\beta^{(r)}}}\right\}$ is a minimal system of generators of $N$. \\
\item $N:_E\m$ is generated by $N$ and $\left\{\cfrac{1}{x^{\alpha_1+1}},\dots,\cfrac{1}{x^{\alpha_k+1}} \right\}$ as an $S$-module. 
\end{enumerate}
\end{lemma}
\begin{proof}
{
(1) Note that $x^{\alpha}\cdot \cfrac{1}{x^{\beta+1}}=0$ if and only if $x^{\alpha}$ does not divide $x^{\beta}$. 
Thus $\cfrac{1}{x^{\beta+1}}$ is contained in $N$ if and only if $x^{\beta} \not\in M$. 
As $N$ is generated by terms, the assertion holds. 
}

{
(2) Let $I_1,\dots,I_r \subset S$ be ideals, and $I=I_1\cap\dots\cap I_r$. By taking the Matlis dual of the natural injection 
$0\to S/I\to \bigoplus_{i=1}^r S/I_i$, 
we have a surjection $\bigoplus_{i=1}^r (S/I_i)^\vee\to (S/I)^\vee\to 0$. Thus $(S/)^\vee=\sum_{i=1}^r(S/I_i)^\vee$. 
}

{
Therefore, $(S/M)^\vee=\sum_{i=1}^r(S/M_i)^\vee$ where $M_i=\left\langle x_1^{\beta^{(i)}_1},\dots,x_n^{\beta^{(i)}_r}\right\rangle$. 
By (1), it is easy to show that $(S/M_i)^\vee$ is generated by $\cfrac{1}{x^{\beta^{(i)}}}$. 
}

{
(3) It is trivial that $N\subset N:\m$. 
By (1), it follows that $\m \cfrac{1}{x^{\alpha_i+1}} \in N$, 
and thus $\cfrac{1}{x^{\alpha_i+1}}\in N:\m$ for all $1\le i\le k$. 
Take $\alpha\in \Z_{\ge 0}^n$ such that $\cfrac{1}{x^{\alpha+1}}\in N:\m$ and $\cfrac{1}{x^{\alpha+1}}\not\in N$. 
As $\cfrac{1}{x^{\alpha+1}}\not\in N$, we have $x^\alpha \in M$ by (1). 
Since $x_i\cdot \cfrac{1}{x^{\alpha+1}} \in N$, $x_i$ does not divide $x^\alpha$, or $x_i$ divides $x^\alpha$ and $x^\alpha/x_i \not\in M$. 
This shows that $x^\alpha$ is a member of the minimal system of monomial generators of $M$. 
}
\end{proof}
For an $S$-submodule $N\subset E$, we define two term modules. 
\begin{definition}
We denote by $T_1(N)$  the submodule of $N$ generated by terms contained in $N$. 
\end{definition}
\begin{definition}
We denote by $T_2(N)$ the submodule of $E$ generated by terms appearing in $N$.
\end{definition}
We note that $T_1(N)\subset N \subset T_2(N)$, 
and $T_1(N)$ is the maximal $S$-submodule of $E$ contained in $N$ generated by terms of $E$, 
and $T_2(N)$ is the minimal $S$-submodule of $E$ containing $N$ generated by terms of $E$. 
For an ideal $J\subset S$, let $M_J$ be the minimal monomial ideal containing $J$. 
Then $(S/M_J)^\vee=T_1(N)$. 
{In the rest part of this paper, we fix a term order $\prec$ on $E$. }
\begin{theorem}\label{composition series}
Let $N\subset E$ be an $S$-module of finite length. 
Let $N_0=T_1(N)$, and define $\xi_i$ and $N_i$ inductively as follows: Choose a $\xi_i\in  N$ satisfying 
\[
\LT_\prec(\xi_i)=\min_\prec\{ \LT_\prec(\eta) \mid \eta\in  N,~\LT_\prec(\eta)\not\in \LT_\prec(N_{i-1}) \}, 
\]
and set $N_i=N_{i-1}+S\xi_i$. 
Then we obtain a sequence of $S$-modules
\[
N_0=T_1(N)\subset N_1 \subset N_2\subset \dots \subset N_{k-1}\subset N_k=N.   
\]
The following hold. 
\begin{enumerate}[{(}1{)}]
\item $N_{i}/N_{i-1} \cong K$ for all $1\le i\le k$. 
\item $\LT_\prec(\xi_0)\prec\LT_\prec(\xi_1)\prec\dots\prec\LT_\prec(\xi_{k-1})\prec\LT_\prec(\xi_k)$. 
\item $N_i$ does not depend on the choice of $\xi_i$. 
\item $\xi_i$ can be chosen so that any term of $\xi_i$ is not in $\LT_\prec(N_{i-1})$, and such $\xi_i$ is unique up to multiplication by $K^\times$. 
\end{enumerate}
\end{theorem}
\begin{proof}
First, we claim that $N_{i-1}$ contains the set 
\[
A_{i-1}:=\{\eta \mid \eta \in N,~\LT_\prec(\eta)\prec \LT_\prec(\xi_i) \}
\]
for $1\le i\le k$. 
Assume to the contrary that $A_{i-1} \not\subset N_{i-1}$. 
Let $\eta\in A_{i-1}$ with $\LT_\prec(\eta)$ minimal among those $\eta\not\in N_{i-1}$. 
Then, by the choice of $\xi_i$, it holds that $\LT_\prec(\eta)\in \LT_\prec(N_{i-1})$. 
Take $\eta'\in N_{i-1}$ such that {$\LM_\prec(\eta)=\LM_\prec(\eta')$} and let $\zeta =\eta-\eta'$. 
Then $\zeta\neq 0$, $\zeta\in N$, $\zeta\not\in N_{i-1}$, and $\LT_\prec(\zeta)\prec \LT_\prec(\eta)$. 
As $\zeta \in A_{i-1}$, this contradicts the minimality of the leading term of $\eta$. 
Thus the claim is proved. 

Since $\LT_\prec(x_j\xi_i)\prec \LT_\prec(\xi_i)$, we have $x_j\xi_i \in A_{i-1}\subset N_{i-1}$ for any $j$. 
Thus $\m\xi_i \in  N_{i-1}$ which shows (1). 
To prove (2), assume to the contrary that $\LT_\prec(\xi_{i})\succ\LT_\prec(\xi_{i+1})$ for some $i$. 
{Note that $\LT_\prec(\xi_{i})\neq \LT_\prec(\xi_{i+1})$ by the definition of $\xi_i$'s. } 
Then $\xi_{i+1} \in A_{i-1} \subset N_{i-1}$ which contradicts to $\xi_{i+1}\not\in N_i$. Thus (2) holds true. 
We will prove (3). 
Let $\xi\in E$ be an element satisfying the same conditions as $\xi_i$ and 
$\xi_i \neq\xi$. {We may assume that $\LC_\prec(\xi_i)=\LC_\prec(\xi)$}. Then $\xi-\xi_i\in  N$ and $\LT_\prec(\xi-\xi_i)\prec \LT_\prec(\xi_i)$ as {$\LM_\prec(\xi_i)=\LM_\prec(\xi)$}. 
Hence $\xi-\xi_i\in N_{i-1}$, 
and thus $N_{i-1}+S\xi_i=N_{i-1}+S\xi$. 
(4) holds true by Gaussian elimination. 
\end{proof}
\begin{definition}
Let $J\subset S$ be an ideal with a system of generators $F=(f_1,\dots, f_m)$, $f_i\in S$,  
and $\Gamma=(\tau_1,\dots,\tau_p)$ a sequence of terms $E$ with $\tau_1\succ \dots \succ \tau_p$. 
For $1\le i\le m$,  
Let $s^{(i)}_1,\dots,s^{(i)}_{q_i}$ be all terms appearing in $\{f_i\tau_j \mid 1\le j\le p\}$. 
We define $M^{(i)}_{F,\Gamma}\in K^{q_i\times p}$ as a $q_i\times (p+1)$ matrix whose $(j,k)$-entry is the coefficient of $s^{(i)}_j$ in $f_i\tau_k$. 
We define $M_{F,\Gamma}$  as the block matrix 
$\left(\begin{array}{c}
M^{(1)}_{F,\Gamma}\\
\vdots\\
M^{(m)}_{F,\Gamma}
\end{array}
\right)$. 
\end{definition}
For $\Gamma=(\tau_1,\dots,\tau_p)$ and $c=(c_1,\dots,c_p)^T$ where ${v}^T$ denotes the transpose of the vector $v$, 
let $\eta=\Gamma \cdot c=\sum_i^p c_i \tau_i$. 
Then $M^{(i)}_{F,\Gamma}\cdot c=0$ if and only if $f_i\eta=0$. 
Hence $M_{F,\Gamma}\cdot c=0$ if and only if $J\eta=0$. 

We will give an algorithm for computing the Matlis dual of $S/J$ where $J$ is an $\m$-primary ideal (Algorithm \ref{alg: length}). 
We omit describing the details of computations for term modules. By Lemma \ref{monomial case}, they are computable by solving some combinatorial problems. 
\begin{algorithm}
\caption{Algorithm for computing $(S/J)^\vee$}
\label{alg: length}
\begin{algorithmic}[1]
\Require A system of generators $F=(f_1,\dots,f_m)$ of an $\m$-primary ideal $J\subset S$.
\Ensure A $K$-basis of $(S/J)^\vee$. 
\State Compute $B=\left\{ \cfrac{1}{x^{\alpha+1}} ~\middle|~ J\cfrac{1}{x^{\alpha+1}} =0 \right\}$, and $N\gets \langle B\rangle$. 
\State $L_1 \gets \emptyset$, 
$L_2 \gets \left\{ \cfrac{1}{x^{\alpha+1}} ~\middle|~ \cfrac{1}{x^{\alpha+1}} \in N:_E\m,~ \cfrac{1}{x^{\alpha+1}}\not\in N \right\}$. 
\While{$L_1\neq L_2$}
\State $\tau_0 \gets \min_{\prec}(L_2 \backslash L_1)$. 
\State $\Gamma \gets (\tau_0,\tau_1,\dots,\tau_p)$ where $\tau_1\succ \dots \succ \tau_p$ and 
\[
\{\tau_1,\dots,\tau_p\}=\left\{ \cfrac{1}{x^{\alpha+1}} ~\middle|~ 
\cfrac{1}{x^{\alpha+1}} \in T_2(N):_E\m,~ \cfrac{1}{x^{\alpha+1}} \not\in \LT_\prec(N),~ \cfrac{1}{x^{\alpha+1}} \prec \tau_0 \right\}. 
\]
\If{$M_{\Gamma,F}\cdot c=0$ for some $0\neq c=(c_0,c_1,\dots,c_p)^T\in K^{p+1}$}
\State $B\gets B\cup\{\xi\}$, $N\gets \langle N, \xi\rangle$ where $\xi:=\Gamma \cdot c=\sum c_i\tau_i$. 
\State $L_2 \gets \left\{ \cfrac{1}{x^{\alpha+1}} ~\middle|~ \cfrac{1}{x^{\alpha+1}} \in \LT_\prec(N):_E\m,~ \cfrac{1}{x^{\alpha+1}}\not\in \LT_\prec(N) \right\}$. 
\Else ~$L_1 \gets L_1\cup \left\{{\tau_0}\right\}$. 
\EndIf
\EndWhile 
\State \Return $B$
\end{algorithmic}
\end{algorithm}
\begin{proof}[Proof of the correctness of Algorithm \ref{alg: length}]
This algorithm is essentially same as the algorithm in \cite{TNN}. 
Here, we give a proof 

Since $S$ is a Noetherian ring and $(S/I)^\vee$ is a Noetherian $S$-module, 
the ascending condition holds for ideals of $S$ and submodules of $(S/I)^\vee$. 
In each step of the while loop, the monomial ideal $\langle x^\alpha \mid \frac{1}{x^{\alpha+1}} \in L_1\rangle \subset S$ or the $S$-submodule $N$ of $(S/I)^\vee$ 
become larger, this algorithm terminate in finite time. 

We will prove that $N$'s and $\xi$'s computed in this algorithm are $N_i$'s and $\xi_i$'s in Theorem \ref{composition series} provided $N=(S/J)^\vee$. 
The module $N$ in line 1 is $T_1((S/J)^\vee)$. 
Thus $0$-th $N$ coincides with $N_0$. 
Assume that $(i-1)$-th $N$ is obtained in the algorithm and it coincides with $N_{i-1}$. 
Since $\xi_i \in N_{i-1}$, we have $\LT_\prec(\xi_i) \in \LT_\prec(N_{i-1}):_E\m$. 
As $\xi_i\in N_{i-1}\subset T_2(N_{i-1})$, the terms of $\xi_i$ except for $\LT_\prec(\xi_i)$ is in 
\[
\left\{ \cfrac{1}{x^{\alpha+1}} ~\middle|~ 
\cfrac{1}{x^{\alpha+1}} \in T_2(N_{i-1}):_E\m,~ \cfrac{1}{x^{\alpha+1}} \not\in \LT_\prec(N),~ \cfrac{1}{x^{\alpha+1}} \prec \LT_\prec(\xi_i) \right\}. 
\]
As $L_1$ is the list of terms which are less than $\LT_\prec(\xi_{i-1})$ or proven not to be the leading term of $\xi_i$, 
$\min_\prec L_2\backslash L_1$ is the candidates of $\LT_\prec(\xi_i)$. 
Thus there exists $c$ satisfying the condition in line 6, the $\xi$ in line 7 should be $\xi_i$, and the new $N$ is $N_i$. 
Therefore, by induction, this algorithm construct the composition series in Theorem \ref{composition series}. 
Hence the output is $(S/J)^\vee$. 
\end{proof}
\section{Algorithm}
\subsection{An algorithm for computing Hilbert--Samuel multiplicities and reductions}
Let $I\subset S=K\llbracket x\rrbracket=K\llbracket x_1,\dots,x_n\rrbracket$ be an ideal with a system of generators $g_1,\dots, g_r$.  
We assume that $R=S/I$ is Cohen--Macaulay. {We denote by $\n$ the unique maximal ideal of $R$}. 
Let $f_1,\dots,f_m \in S$ such that $J:=\langle f_1,\dots,f_m\rangle_R \subset R$ {is} an $\n$-primary ideal. 

Let $\underline{t}=(t_{ij})_{1\le i\le d,d+1\le j\le m}$ where $t_{ij}$'s are indeterminates over $K$. 
Let $\K=K(t_{i,j}\mid i,j)$, $S'=\K\llbracket x_1,\dots,x_n\rrbracket$, and $R'=S'/IS'$. 
We set 
\begin{eqnarray*}
Q_{\underline{t}} &:=& \left\langle {f}_i+\sum_{j=d+1}^m t_{ij}{f}_j ~\middle|~ 1\le i\le d\right\rangle_{R'} \subset R', \\
P_{\underline{t}} &:=& \left\langle f_i+\sum_{j=d+1}^m t_{ij}f_j ~\middle|~ 1\le i\le d\right\rangle + IS'\subset S'. 
\end{eqnarray*}
For $\underline{a}=(a_{ij})_{i,j}  \in K^{d\times (m-d)}$, we set 
\[
Q_{\underline{a}} := \left\langle {f}_i+\sum_{j=d+1}^m a_{ij}{f}_j ~\middle|~ 1\le i\le d\right\rangle_{R} \subset R. 
\]
Since $Q_{\underline{t}}$ is a image of $P_{\underline{t}}$ in $R'$, we have $R'/Q_{\underline{t}}=S/P_{\underline{t}}$.

We will give an algorithm for computing Hilbert--Samuel multiplicities and reductions (Algorithm \ref{alg: main}). 
\begin{algorithm}\label{main alg}
\caption{Algorithm for computing $e_R(J)$ and reductions of $J$}
\label{alg: main}
\begin{algorithmic}[1]
\Require $\m$-primary ideal $J=\langle f_1,\dots,f_m\rangle_R \subset R$, and an ideal $I=\langle g_1,\dots,g_r\rangle \subset S$ such that $R/I$ is Cohen--Macaulay.
\Ensure $e_R(J)$, PolyList, MatList. 
\State  $\mathrm{PolyList}\gets \emptyset$, $ \mathrm{MatList}\gets \emptyset$. 
\State $F \gets ({f}_1+\sum_{i=d+1}^m t_{i1}{f}_i,\dots,{f}_d+\sum_{i=d+1}^m t_{id}{f}_i,g_1,\dots,g_r)$. 
\State Compute $B=\left\{ \cfrac{1}{x^{\alpha+1}} ~\middle|~ J\cfrac{1}{x^{\alpha+1}} =0 \right\}$, and $N\gets \langle B\rangle$. 
\State $L_1 \gets \emptyset$, 
$L_2 \gets \left\{ \cfrac{1}{x^{\alpha+1}} ~\middle|~ \cfrac{1}{x^{\alpha+1}} \in N:_E\m,~ \cfrac{1}{x^{\alpha+1}}\not\in N \right\}$. 
\While{$L_1\neq L_2$}
\State $\tau_0 \gets \min_{\prec}(L_2 \backslash L_1)$. 
\State $\Gamma \gets (\tau_0,\tau_1,\dots,\tau_p)$ where $\tau_1\succ \dots \succ \tau_p$ and 
\[
\{\tau_1,\dots,\tau_p\}=\left\{ \cfrac{1}{x^{\alpha+1}} ~\middle|~ 
\cfrac{1}{x^{\alpha+1}} \in T_2(N):_E\m,~ \cfrac{1}{x^{\alpha+1}} \not\in \LT_\prec(N),~ \cfrac{1}{x^{\alpha+1}} \prec \tau_0 \right\}. 
\]
\If{$M_{\Gamma,F}\cdot c=0$ for some $0\neq c=(c_0(\underline{t}), c_1(\underline{t}),\dots,c_p(\underline{t}))^T\in K[\underline{t}]^{p+1}$}
\State $B\gets B\cup\{\xi\}$, $N\gets \langle N, \xi\rangle$ where $\xi:=\Gamma \cdot c=\sum c_i\tau_i$. 
\State $\mathrm{PolyList} \gets \mathrm{PolyList} \cup~\{{c_0(\underline{t})}\}$. 
\State $L_2 \gets \left\{ \cfrac{1}{x^{\alpha+1}} ~\middle|~ \cfrac{1}{x^{\alpha+1}} \in \LT_\prec(N):_E\m,~ \cfrac{1}{x^{\alpha+1}}\not\in \LT_\prec(N) \right\}$. 
\Else ~$L_1 \gets L_1\cup \left\{{\tau_0}\right\}$, $ \mathrm{MatList}\gets  \mathrm{MatList} \cup~\{M_{\Gamma,F}\}$. 
\EndIf
\EndWhile 
\State \Return $|B|$, PolyList, MatList. 
\end{algorithmic}
\end{algorithm}
\begin{theorem}
\label{thm: main}
\begin{enumerate}[{(}1{)}]
\item Algorithm \ref{alg: main} computes the Hilbert--Samuel multiplicity $e_R(J)=\abs{B}$ of $J\subset R$. 
\item For $\underline{a}=(a_{ij})_{i,j}\in K^{d\times (m-d)}$ satisfying the following two conditions, 
$Q_{\underline{a}}$ is a reduction of $J$. 
\begin{enumerate}[{(}a{)}]
\item $c(\underline{a})\neq 0$ for all $c(\underline{t}) \in \mathrm{PolyList}$. 
\item For any $M(\underline{t})\in \mathrm{MatList}$, the linear equation $M(\underline{a})y=0$ has no non-trivial solution. 
Here, $M(\underline{a})$ is the matrix obtained by substituting $\underline{a}$ to $\underline{t}$ in $M(\underline{t})$. 
\end{enumerate}
\end{enumerate} 
\end{theorem}
\begin{proof}
{For} $\underline{b}=(b_{ij})_{i,j}  \in K^{d\times m}$, we set 
\[
Q'_{\underline{b}} := \left\langle \sum_{i=1}^m b_{ij}{f}_i ~\middle|~ 1\le i\le d\right\rangle_{R} \subset R. 
\]
By Theorem \ref{red}, there exists $\b \subset K[t_{ij}\mid 1\le i \le d, 1\le j\le m]$ 
such that $Q'_{\underline{b}}$ is a reduction of $J$ for $\underline{b}\in U_1:=K^{d\times m}\backslash V_K(\b)$. 
Let $\a\subset K[t_{ij} \mid 1\le i \le d, m-d \le j\le m] $ be the ideal obtained from $\b$ by substituting $t_{ii}=1, t_{ij}=0$ for $1\le i, j \le d, i\neq j$. 
Then $Q_{\underline{a}}$ is a reduction of $J$ for ${\underline{a}}\in U_2:=K^{d\times (m-d)} \backslash V_K(\a)$. 
We will show that $U_2 \neq \emptyset$. 
Since any two Zariski open sets of $K^{d\times m}$ has a non-empty intersection, 
there exists $\underline{b} \in U_1$ such that the $d\times d$ matrix $(b_{ij})_{1\le i,j\le d}$ has a non-zero determinant. 
If this is the case, $Q'_{\underline{b}}=Q_{\underline{a}}$ for some $\underline{a}\in K^{d\times (m-d)}$, 
and thus $U_2$ is not empty. 
By Corollary \ref{CM case}, for generic $\underline{a}\in  K^{d\times (m-d)}$, 
we have 
\[
e_R(J)=\ell(R/Q_{\underline{a}})=\ell_{R'}(R'/Q_{\underline{t}})=\ell_{S'}(S'/P_{\underline{t}}).
\]
Since Algorithm \ref{alg: main} computes $\ell(S'/P_{\underline{t}})=\abs{B}$ using Algorithm \ref{alg: length}, 
this proves (1). 

Assume that $\underline{a}$ satisfies (a) and (b) in (2). 
We apply Algorithm \ref{alg: length} to $P_{\underline{t}}$, and consider each step with substituting $\underline{a}$ to $\underline{t}$. 
When a new $\xi$ is added to $B$, the leading coefficient of $\xi$ does not vanish by the condition (a), 
and the condition that $P_{\underline{t}}$ annihilates $\xi$ is preserved under the substitution. 
When $M_{\Gamma,F}c=0$ does not have a non-trivial solution, 
this condition is also preserved under the substitution by the condition (b) and since the set $\{\tau_1,\dots,\tau_p\}$ does not change or becomes smaller.  
Therefore $\ell(R/Q_{\underline{a}})=\ell(S/P_{\underline{a}})=e_R(J)$. 
Since $R$ is Cohen--Macaulay, $Q_{\underline{a}}$ is a reduction of $J$, which proves (2). 
\end{proof}
We note that the set of $\underline{a} \in \mathrm{PolyList}$ satisfying (a) and (b) forms a Zariski open set. 
Thus conditions (a) and (b)  in Theorem \ref{thm: main} together express what generic linear combinations means explicitly. 
If we choose coefficients using random numbers, the coefficients tend to be large. 
We are able to look for a reduction with small coefficients using Theorem \ref{thm: main}. 
Once we obtain a reduction $Q$ of $J$, we may use $Q$ for solving membership problem of $\overline{J}$; 
for $f\in S$, $f\in \overline{J}$ if and only if $e(\langle Q, f\rangle)=\ell(R/Q)$. 

Algorithm \ref{main alg} also give a theoretical result that Hilbert--Samuel multiplicities are computable 
even if ideals are generated by power series with infinitely many terms. 
Recall that a power series $\sum_{\alpha\in \Z_{\ge 0}} c_\alpha x^\alpha$ is computable 
if the function $\alpha \mapsto c_\alpha$ is computable. 
\begin{corollary}\label{infinite terms}
Let $I \subset S:=K\llbracket x_1,\dots,x_n \rrbracket$ be an ideal and $R=S/I$. Let $\m$ be the maximal ideal of $R$ and $J\subset R$. 
Suppose that $R$ is Cohen--Macaulay and $J$ is $\m$-primary ideal. 
If $I$ and $J$ are generated by computable power series, then $e_R(J)$ is computable using Algorithm \ref{main alg}. 
\end{corollary}
\begin{proof}
If $\eta$ is a computable power series, $\eta \cfrac{1}{x^{\alpha+1}}$ can be computed in finite time since 
there exist only finitely many terms $x^\beta$ of $S$ such that $x^\beta \cfrac{1}{x^{\alpha+1}}\neq 0$. 
Therefore, we are able to execute each step of Algorithm \ref{main alg} in finite time in this case. 
\end{proof}
\begin{remark}\label{rmk:infinite terms}
If $I$ and $J$ are generated by power series with infinitely many terms, existing methods for computing $e_R(J)$ using Buchberger's type algorithms 
(e.g. Gr\"obner basis, Mora's algorithm \cite{Mora}, Lazard's homogenization method \cite{Lazard}) are not applicable.  
{Because the computability of Buchberger's type algorithms in formal power series rings are essentially based on the fact that if an ideal of 
$S=K\llbracket x_1,\dots,x_n \rrbracket$ is generated by polynomials $f_1,\dots,f_r \in P=K[x_1,\dots,x_n]$ 
then the syzygy module $\Syz_S(f_1,\dots,f_r)$ of $f_1,\dots,f_r$ over $S$ is generated by the syzygy module 
$\Syz_P(f_1,\dots,f_r)$ of $f_1,\dots,f_r$ over the polynomial ring $P$ which is computable. 
}

We note that it is known that whether $J$ is $\m$-primary or not is not decidable. 
Let $P(y_1,\dots,y_r)\in \Z[y_1,\dots,y_r]$ be a polynomial, and $p_1=2,{p_2}=3,\dots,p_r$ the first $r$ prime numbers. 
For $a \in \Z_{\ge 0}$, we set 
\[
c_a=
\left\{
\begin{array}{cc}
1 & \mbox{ if } a=p_1^{m_1}\cdots p_r^{m_r} \mbox{ and } P(m_1,\dots,m_r)=0 \\
0 & \mbox{ otherwise}
\end{array}
\right.
\]
and let $f(x_1)=\sum_a c_a x_1^a$. Then $f$ is computable, and $f\neq 0$ if and only if the equation $P(y_1,\dots,y_r)=0$ has a non-negative integer solution. 
It is known that there is no algorithm for deciding this. 
Therefore, there is no algorithm for deciding whether an ideal $\langle f(x_1) \rangle\subset K\llbracket x_1\rrbracket$ 
is $\m$-primary or not. 

{
Cohen--Macaulay of $R=S/I$ is also not decidable. Let $f(x_1)$ be a computable power series. 
The residue ring $K\llbracket x_1,x_2,x_3 \rrbracket/\langle x_1^2f(x_1), x_1x_2f(x_1) \rangle$ is Cohen--Macaulay if and only if $f(x_1)=0$ which is not decidable in general. 
So we consider only the case where we know that $R$ is Cohen--Macaulay. 
For example, complete intersection rings such as $S/\langle f \rangle$ with $f\neq 0$, and the completion of Cohen--Macaulay rings of essentially finite type over $K$. }
\end{remark}
\subsection{Modulo $p$ method}
The main part of the algorithms presented in this paper is solving linear equations. 
By using so-called modular method, one can reduce computational time of this part. 
For simplicity, assume that $K=\Q$ and the coefficients of $f_i$'s and $g_j$'s are in $\Z$. 
Then the entries of matrices $M_{\Gamma,F}$ appearing in Algorithm \ref{alg: main} is in the polynomial ring $\Z[t_{ij} \mid 1\le i\le d,d+1\le j\le m]$. 
Applying Gaussian elimination to such matrices is a hard task in general. 
By Theorem \ref{composition series} (4), the dimension of the solution space of $M_{\Gamma,F}x=0$ is at most $1$. 
Here, we give a effective method for solving linear equations with matrices of this type. 

Let $R=\Z[t]=\Z[t_1,\dots,t_r]$, and $K=\Q(t_1,\dots,t_r)$ the fractional field of $R$. 
Let $M\in R^{n \times m}$ be an $n\times m$ matrix. 
Let $p$ be a prime number and $0\le a_1,\dots,a_n \le p-1$. 
Then $\p=\langle p, t_1-a_1,\dots, t_n-a_n\rangle$ is a maximal ideal of $R$, and $R/\p=\F_p$. 
We denote $M \mod \p$ by $\overline{M}\in \F_p^{n\times m}$. 
Let $x=(x_1,\dots,x_m)^T$ be a vector of indeterminates. 
For a vector $c=(c_1,\dots,c_m)^T$, we call $\supp(c):=\{i\mid c_i\neq 0\}$ the support of $c$. 
For a subset $\Lambda \subset \{1,\dots,n\}$, $[M]_\Lambda$ denotes the submatrix of $M$ 
consisting of the $i$-th columns of $M$ for $i\in \Lambda$. 
\begin{proposition}\label{mod p}
Let the notation be as above. 
Assume that the dimension of the solution space in $K^m$ of the linear equations $Mx=0$ is at most $1$. 
\begin{enumerate}[{(}1{)}]
\item If the linear equation $\overline{M}x=0$ does not have non-trivial in $\F_p^m$, then $Mx=0$ does not have non-trivial in $K^m$. 
\item Assume that $\overline{M}c=0$ for some $c\in \F_p^m$. If $\p$ is generic, that is, $p$ is sufficiently large and $a_1,\dots,a_r$ are generically chosen, 
then the linear equation $[M]_{\supp(c)} x'=0$ has a solution in $K^{\abs{\supp(c)}}$ where $x'=(x_i \mid i\in \supp(c))$. 
\end{enumerate}
\end{proposition}
\begin{proof}
We will prove the contraposition of (1). 
Assume that $Mx=0$ has a non-trivial, which is equivalent to that any $m$-minor of $M$ is zero. 
Then any $m$-minor of $\overline{M}$ is also zero, and thus  $\overline{M}x=0$ has a non-trivial solution. 

Let $I_i(M)\subset R$ be the ideal generated by all $i$-minors of $M$ for $i=m-1,m$. 
Since the dimension of the solution space in $K^m$ of the linear equations $Mx=0$ is at most $1$, we have $I_{m-1}(M)\neq 0$. 
If $\p$ is generic, $I_{m-1}(\overline{M})\neq 0$, 
and thus the dimension of the solution space in $\F_p^m$ of the linear equations $\overline{M}x=0$ is also at most $1$. 
Since $\overline{M}x=0$ has a non-trivial solution, $I_{m}(\overline{M})=0$. 
If $\p$ is generic, this implies $I_m(M)=0$. 
Then there exists $0\neq c(t)=(c_1(t),\dots,c_m(t))^T\in \Z[t]^m$ which is a $K$-basis of the solution space of $Mx=0$. 
Let $c= c(t) \mod \p$. As $\p$ is generic, $\supp(c(t))=\supp(c)$, and $c$ is a $\F_p$-basis of the solution space of $\overline{M}x=0$. 
This proves the assertion of (2). 
\end{proof}
If  $[M]_{\supp(c)} x'=0$ in (2) has no non-trivial solution, we should take another $\p$ or solve $Mx=0$ directly, 
but it rarely happens. 
Using Proposition \ref{mod p}, we can reduce the computation time of Algorithm \ref{alg: main}. 
\section{{Examples}}
We implement Algorithm \ref{main alg} in a computer algebra system Risa/Asir. We give some examples of computation. 

We fix a term order $\prec$ on $E$ as follows: 
Let $\prec_{\mathrm{glex}}$ be the graded lexicographic order with $z\prec_{\mathrm{glex}} y\prec_{\mathrm{glex}} x$, and we define $\prec$ by 
\[
\frac{1}{x^{i_1+1}y^{i_2+1}z^{i_3+1}} \prec \frac{1}{x^{j_1+1}y^{j_2+1}z^{j_3+1}} \mbox{ if } 
x^{i_1}y^{i_2}z^{i_3} \prec_{\mathrm{glex}} x^{j_1}y^{j_2}z^{j_3}. 
\]
\begin{example}
Let $J=\langle x^3,y^2,xy \rangle \subset R=\C\llbracket x,y\rrbracket$. 
Then $e_R(J)=\ell(R/\langle x^3+axy+y^2+bxy\rangle)$ for generic $a,b \in \C$. 
Algorithm \ref{main alg} computes $M:=(R/\langle x^3+axy+y^2+bxy\rangle)^\vee$ for generic $a,b \in \C$. 
We will execute Algorithm \ref{main alg} by hand to illustrate how this algorithm works. 

At line 1, we compute $N=T_1(M)=(R/\langle x^3, y^2, xy\rangle)^\vee \subset E_R$, and $B=\left\{\frac{1}{xy},\frac{1}{x^2y},\frac{1}{x^3y},\frac{1}{xy^2} \right\}$. 
Since $N$ is generated by monomials, $N=T_1(N)=T_2(N)$. 
At line 2, $L_1=\emptyset$, and by Lemma {monomial case} (3), $L_2=\left\{\frac{1}{xy^3},\frac{1}{x^2y^2},\frac{1}{x^4y}\right\}$. 

In the while loop beginning form line 3, we construct $\xi_i$'s in Theorem \ref{composition series}. 

The first candidate of the leading term of $\xi_1$ is $\tau_0:=\min_\prec{L_2}=\frac{1}{xy^3}$. 
The set $\{\tau_1,\dots,\tau_p\}$ in line 7 is the empty set in this case. 
Since $J\frac{1}{xy^3}\neq 0$, we update $L_2$ with $\left\{\frac{1}{xy^3}\right\}$. 
As $L_1\neq L_2$, we go to line 6. 

The next candidate of the leading term of $\xi_1$ is $\tau_0:=\min_\prec{L_2 \backslash L_1}=\frac{1}{x^2y^2}$. 
The set $\{\tau_1,\dots,\tau_p\}$ in line 7 is $\{\frac{1}{xy^3}\}$ in this case. 
Since there is no non-trivial $c_0,c_1\in \C$ satisfying $J\left(c_0\frac{1}{x^2y^2}+c_1\frac{1}{xy^3}\right)=0$, we update $L_2$ with $\left\{\frac{1}{xy^3}, \frac{1}{x^2y^2}\right\}$. 
As $L_1\neq L_2$, we go to line 6. 

The next candidate of the leading term of $\xi_1$ is $\tau_0:=\min_\prec{L_2 \backslash L_1}=\frac{1}{x^4y}$. 
The set $\{\tau_1,\dots,\tau_p\}$ in line 7 is $\{\frac{1}{xy^3}, \frac{1}{x^2y^2}\}$ in this case. 
We will look for $c_0,c_1,c_2\in \C$ satisfying $J\left(c_0\frac{1}{x^4y}+c_1\frac{1}{x^2y^2}+c_2\frac{1}{xy^3}\right)=0$, 
and find a non-trivial solution $(c_0,c_1,c_2)=(a,b,-1)$. Thus 
\[
\xi_1 = \cfrac{a}{x^4y}+\cfrac{b}{xy^3}+\cfrac{-1}{x^2y^2}. 
\]
We update $B$ and $N$ with $\left\{\frac{1}{xy},\frac{1}{x^2y},\frac{1}{x^3y},\frac{1}{xy^2},\xi_1 \right\}$ and $N+\langle \xi_1\rangle$ respectively. 
Then 
\[
\LT_\prec(N)=(R/\langle x^3, y^2, xy\rangle)^\vee+\langle \LT_\prec(\xi_1)\rangle=(R/\langle x^4, y^2, xy\rangle)^\vee, 
\]
and thus we update $L_2$ to $\left\{\frac{1}{xy^3},\frac{1}{x^2y^2},\frac{1}{x^5y}\right\}$. 
We note that $T_2(N)=(R/\langle x^3, y^2, xy\rangle)^\vee+\left\langle \cfrac{1}{x^4y}, \cfrac{1}{xy^3},\cfrac{1}{x^2y^2}\right\rangle$. 
Then, as $L_1\neq L_2$, we go to line 6. 

The first candidate of the leading term of $\xi_2$ is $\tau_0:=\min_\prec{L_2 \backslash L_1}=\frac{1}{x^5y^2}$. 
The set $\{\tau_1,\dots,\tau_p\}$ in line 7 is $\{\frac{1}{xy^3}, \frac{1}{x^2y^2},\frac{1}{xy^4}, \frac{1}{x^2y^3},\frac{1}{x^3y^2}\}$ in this case. 
Since there is no non-trivial $c_1,c_2,c_3,c_4,c_5\in \C$ satisfying 
\[
J\left(c_0\frac{1}{x^5y}+c_1\frac{1}{xy^3}+c_2\frac{1}{x^2y^2}+c_3\frac{1}{xy^4}+c_4\frac{1}{x^2y^3}+c_5\frac{1}{x^3y^2}\right)=0, 
\]
we update $L_2$ with $\left\{\frac{1}{xy^3}, \frac{1}{x^2y^2}, \frac{1}{x^4y^2}\right\}$. 
Then we have $L_1=L_2$ and thus we leave the while loop, 
and conclude that $e_R(J)=\abs{B}=\left|\left\{\frac{1}{xy},\frac{1}{x^2y},\frac{1}{x^3y},\frac{1}{xy^2},\xi_1 \right\}\right|=5$. 
\end{example}
\begin{example}
Let $J=\langle f_1,f_2,f_3,f_4 \rangle\subset R=\C\llbracket x,y,z\rrbracket$ where 
\[
f_1= x^2+y^3+z^3,~f_2=y^3+xz^3,~f_3=z^4+xy^3,~f_4=x^2+xyz+y^4
\] 
Then $e_R(J)=\ell(R/\langle f_1+af_4,f_2+bf_4,f_3+cf_4 \rangle)$ for generic $a,b,c\in\C$. 

Algorithm \ref{main alg} computes $N:=(R/\langle  f_1+af_4,f_2+bf_4,f_3+cf_4 \rangle)^\vee \subset E_R$ for generic $a,b,c\in \C$ as follows: 
\[
T_1(N)=(R/\langle x^2,y^3,z^3,xyz \rangle)^\vee
\]
has length $14$, and $N$ is spanned over $\C$ by $T_1(N)$ and the following four elements 
\begin{align*}
& \cfrac{1}{z^2y^2x^2}+\cfrac{-1}{zyx^3}+\cfrac{1}{z^4yx},~~
\cfrac{c}{z^3y^2x^2}+\cfrac{-1}{zyx^3}+\cfrac{a-b+1}{z^4yx}+\cfrac{b}{zy^4x}+\cfrac{-c}{z^2yx^3}+\cfrac{c}{z^5yx}\\
& \cfrac{1}{z^2y^3x^2}+\cfrac{-1}{zy^2x^3}+\cfrac{1}{z^4y^2x},\\
&\cfrac{c}{z^3y^3x^2}+\cfrac{-b}{zyx^3}+\cfrac{b-c}{z^4yx}+\cfrac{c}{zy^4x}+\cfrac{-1}{zy^2x^3}
+\cfrac{c}{zyx^4}+\cfrac{a-b+1}{z^4y^2x}+\cfrac{b}{zy^5x}+\cfrac{-c}{z^4yx^2}+\cfrac{-c}{z^2y^2x^3}+\cfrac{c}{z^5y^2x}. 
\end{align*}
Thus we have $e_R(J)=\ell(N)=14+4=18$. 
\end{example}
\begin{example}
{
Let $S=\C\llbracket x,y,z\rrbracket$, $R=S/\langle x^2+y^3+z^4\rangle$ and $J_1=\langle x^2,xy,z^2\rangle_R\subset R$. 
Since $R$ is a complete intersection, $R$ is a Cohen--Macaulay local ring of dimension 2. 
Then 
\[
e_R(J_1)=\ell_R(R/\langle x^2+a z^2,xy+bz^2\rangle)=\ell_S(S/\langle x^2+a z^2,xy+bz^2, x^2+y^3+z^4\rangle) 
\]
for generic $a,b\in \C$. 

Algorithm \ref{main alg} computes $N:=(S/\langle x^2+a z^2,xy+bz^2, x^2+y^3+z^4\rangle)^\vee \subset E_S$ as follows: 
\[
T_1(N)=(S/\langle z^2,xy,x^2,y^3\rangle)^\vee
\]
has length $8$, and  
$N$ is spanned over $\C$ by $T_1(N)$ and the following two elements 
\[
\cfrac{a}{zy^4x}+\cfrac{1}{z^3yx}+\cfrac{-b}{zy^2x^2}+\cfrac{-a}{zyx^3},~~~
\cfrac{a}{z^2y^4x}+\cfrac{1}{z^4yx}+\cfrac{-b}{z^2y^2x^2}+\cfrac{-a}{z^2yx^3}. 
\]
Since $\ell(T_1(N))=8$, we have $e_R(J_1)=\ell(N)=8+2=10$. 

Let $J_2=\langle J_1, xz\rangle_R=\langle x^2,xy,z^2,xz\rangle_R\subset R$. 
Then 
\[
e_R(J_1)=\ell_R(R/\langle x^2+a_1z^2+a_2xz,xy+b_1z^2+b2xz\rangle)=\ell_S(S/\langle x^2+a_1z^2+a_2xz,xy+b_1z^2+b_2xz,x^2+y^3+z^4\rangle) 
\]
for generic $a,b\in \C$. 

Algorithm \ref{main alg} computes $N:=(S/\langle x^2+a_1z^2+a_2xz,xy+b_1z^2+b_2xz,x^2+y^3+z^4\rangle)^\vee \subset E_S$ as follows: 
\[
T_1(N)=(S/\langle z^2,xy,x^2,y^3,xz\rangle)^\vee
\]
has length $7$, and $N$ is spanned over $\C$ by $T_1(N)$ and the following three elements 
\begin{align*}
&\cfrac{b_2a_1-b_1a_2}{zy^2x^2}+\cfrac{a_2}{z^3yx}+\cfrac{-a_1}{z^2yx^2},~~~\cfrac{b_2a_1-b_1a_2}{zy^4x}+\cfrac{b_2}{z^3yx}+\cfrac{-b_1}{z^2yx^2}+\cfrac{-b_2a_1+b_1a_2}{zyx^3},\\
&\cfrac{b_2a_1^3-b_1a_2a_1^2}{zy^5x}+\cfrac{-a_1^2+a_2^2a_1}{z^4yx}+\cfrac{b_2a_1^2+(-b_2a_2^2-2b_1a_2)a_1+b_1a_2^3}{z^3y^2x}+\cfrac{-a_2a_1^2}{z^3yx^2}
+\cfrac{(b_2a_2+b_1)a_1^2-b_1a_2^2a_1}{z^2y^2x^2}\\
&+\cfrac{(-b_2^2a_2-2b_2b_1)a_1^2+(2b_2b_1a_2^2+2b_1^2a_2)a_1-b_1^2a_2^3}{zy^3x^2}+\cfrac{a_1^3}{z^2yx^3}+\cfrac{-b_2a_1^3+b_1a_2a_1^2}{zy^2x^3}+\cfrac{-a_1^3}{z^2y^4x}. 
\end{align*}
Thus $e_R(\langle J_1, xz\rangle)=7+3=10=e_R(J_1)$, and $xz\in \overline{J_1}$. 
}
\end{example}
Algorithm \ref{main alg} is applicable when $I$ and $J$ are generated by computable power series (Corollary \ref{infinite terms}). 
Our algorithm also has some advantages in the case where ideals are generated by polynomials that have terms of higher degree. 

Suppose that $J\subset R=K\llbracket x_1,\dots,x_n \rrbracket$ is $\m$-primary and generated by polynomials $f_1,\dots,f_r \in P=K[x_1,\dots,x_n]$. 
Let $J'=\langle f_1,\dots,f_r \rangle \subset K[x_1,\dots,x_n]$, and $\{\p_1,\dots, \p_m\}$ the set of associated prime ideals of $J'$ where $\p_1=\langle x_1,\dots,x_n\rangle$. 
Let $J'=\q_1\cap\dots\cap\q_m$ be the primary decomposition of $J'$ with $\sqrt{\q_i}=\p_i$. 
If $m=1$, that is, if $J'$ is a primary ideal, $e_R(J)$ can be computed using the algorithm in \cite{MR}. 
In Macaulay2 \cite{M2}, a command ``multiplicity" is implemented for computing the summation of the Hilbert--Samuel multiplicities 
$\sum_{i=1}^m e_{P_i}(\q_i P_i)$ where $P_i=P_{\p_i}$ is the localization of $P$ at $\p_i$. 
If $m\neq 2$, to apply the algorithm in \cite{MR}, we should compute the primary component $\q_1$ which is not easy when the rest part $\q_2\cap\dots\cap\q_m$ is complicated. 
Even if $m=1$ we should check that $J'$ is actually a primary ideal. 

On the other hand, Algorithm \ref{main alg} does not need to compute $\q_1$ when $J'$ is not primary. 
\begin{example}\label{ex}
{
Let $R=\C\llbracket x,y,z\rrbracket$ and $J=\langle x^2,xyz,y^3,z^4\rangle$. 
Then 
\[
e_R(J)=\ell(R/\langle x^2-az^4,xyz-bz^4,y^3-cz^4\rangle)
\]
for generic $a,b,c\in\C$. By executing Algorithm \ref{main alg}, we see that, for generic $a,b,c\in\C$, the Matlis dual 
$(R/\langle x^2-az^4,xyz-bz^4,y^3-cz^4\rangle)^\vee$ is generated over $R$ by 
\[
\xi=\cfrac{b^2}{z^7y^2x}+\cfrac{-ca^2}{zy^3x^3}+\cfrac{-ba^2}{z^2yx^4}+\cfrac{-c^2a}{zy^6x}
+\cfrac{-cba}{z^2y^4x^2}+\cfrac{-b^2a}{z^3y^2x^3}+\cfrac{ca}{z^5y^3x}+\cfrac{-cb^2}{z^3y^5x}+\cfrac{ba}{z^6yx^2}+\cfrac{-b^3}{z^4y^3x^2}
\]
and $e_R(J)=24$. The computation finished within 1 second on a laptop with Intel Core i5-3320M at 2.6 GHz, with 8.0 GB memory running Windows 10. 

We will consider ideals with high degree terms in their system of generators. 
Let 
\[
F=\{f_1:=x^2+z^{10}+y^{20}+x^{200},f_2:=xyz+x^{10}+xy^{20}+z^{100},f_3:=y^3+x^{10}+y^{100},f_4:=z^4+y^{10}+x^{20}+z^{100}\}.
\]
A primary decomposition of the polynomial ideal generated by $F$ is hard to compute. 
On the other hand, by executing Algorithm \ref{main alg}, we see that 
$N:=(R/\langle f_1-af_4,f_2-bf_4,f_3-cf_4\rangle)^\vee$ coincides with $(R/\langle x^2-az^4,xyz-bz^4,y^3-cz^4\rangle)^\vee$, 
and $e_R(\langle F\rangle)=24$. 
This computation also finished within 1 second. 
The elements of $E_R$ appearing in the computation do not change after replacing the input form $\{x^2,xyz,y^3,z^4\}$ to $F$, 
and they are annihilated by the higher terms in $F$. 
As a consequence, the added higher degree terms does not affect on the computation time. 
}
\end{example}
\section*{Acknowledgement}
This work was supported by JSPS KAKENHI Grant Number JP18K03320.
\bibliographystyle{elsarticle-harv}

\end{document}